\documentclass[12pt,a4paper]{article}

\usepackage{amsmath,amssymb,amsthm}
\usepackage{mathrsfs}
\usepackage{showkeys}

\newtheorem{theorem}{Theorem}
\newtheorem{lemma}[theorem]{Lemma}
\newtheorem{proposition}[theorem]{Proposition}
\newtheorem{definition}[theorem]{Definition}
\newtheorem{corollary}[theorem]{Corollary}

\newtheorem*{theorema}{Theorem A}
\newtheorem*{theoremb}{Theorem B}

\newtheorem*{hypothesis*}{Hypothesis}
\newtheorem*{lemmaa}{Lemma A}

\newtheorem{hypothesis}{Hypothesis}
\newtheorem*{ihypothesis}{Inductive Hypothesis}

\newcommand{\Galph}{G_\alpha}

\newcommand{\spaceP}{\mathcal{P}}

\newcommand{\Fix}{{\rm Fix}}

\begin{document}


\pagestyle{plain}

\pagenumbering{arabic}
\title{Transitive projective planes and $2$-rank}
\author{Nick Gill}
\maketitle

\begin{abstract}
Suppose that a group $G$ acts transitively on the points of a non-Desarguesian plane, $\spaceP$. We prove first that the Sylow $2$-subgroups of $G$ are cyclic or generalized quaternion. We also prove that $\spaceP$ must admit an odd order automorphism group which acts transitively on the set of points of $\spaceP$. \footnote{Comments are welcome and should be sent to: {\tt nickgill@cantab.net}. I would also like to thank Prof Cheryl Praeger for her excellent advice. In particular any virtues in the exposition of this paper are largely due to her.}

\end{abstract}

\section{Introduction}

In 1959 Ostrom and Wagner proved that if a finite projective plane, $\spaceP$, admits a automorphism group which acts 2-transitively on the set of points of $\spaceP$ then $\spaceP$ is Desarguesian \cite{oswa}. It has long been conjectured that the same conclusion holds if the phrase {\it 2-transitively} is replaced by {\it transitively}.

A number of results have appeared which partially prove this conjecture under certain extra conditions. Most notably, in 1987 Kantor proved that if $\spaceP$ has order $x$ and $\spaceP$ admits a group $G$ which acts {\it primitively} on the set of points of $\spaceP$ then either $\spaceP$ is Desarguesian and $G\geq PSL(3,x)$, or else $x^2+x+1$ is a prime and $G$ is a regular or Frobenius group of order dividing $(x^2+x+1)(x+1)$ or $(x^2+x+1)x$ \cite{kantor}.

The results we present in this paper constitute a further advance towards a proof of the conjecture. In particular we prove two main theorems.

\begin{theorema}
Suppose that a group $G$ acts transitively on the set of points of $\spaceP,$ a non-Desarguesian plane. Then $m_2(G)\leq 1$.
\end{theorema}

Here $m_2(G)$ is the $2$-rank of $G$. Thus $m_2(G)=1$ means that the Sylow $2$-subgroups of $G$ are cyclic or generalized quaternion. 

\begin{theoremb}
Suppose a projective plane $\spaceP$ admits an automorphism group which is transitive on the set of points. Then $\spaceP$ admits an odd order automorphism group which is transitive on the set of points
\end{theoremb}

The structure of the paper is as follows. In Section \ref{section:framework} we reduce the proof of Theorem A to a question about subgroups of general linear groups. In particular, at the end of this section we state Lemma A and demonstrate that this lemma implies Theorem A. In Section \ref{S: progress} we give a proof of Lemma A, thereby also proving Theorem A. Finally in Section \ref{S: quaternionsylowtwos} we analyze the situation where $G$ has generalized quaternion Sylow $2$-subgroups and we prove Theorem B.

Note that the methods used in different sections vary considerably and hence so does our notation. We explain our notation at the start of each section or subsection.

\section{A framework to prove Theorem A}\label{section:framework}

Our aim in this section is to set up a framework to prove Theorem A. In order to do this we will split into two subsections. The first subsection outlines some basic group theory results which will be needed in the remainder of the paper. In the second subsection we will apply these results to the projective plane situation; in particular we will state Lemma A, and will demonstrate that Lemma A implies Theorem A.

\subsection{Some background group theory}\label{S: background}

Throughout this section we use standard group theory notation. Note that, for an element $g\in G,$ we write $g^G$ for the set of $G$-conjugates of $g$ in $G$. A cyclic group of order $n$ will sometimes just be written $n$. We write $G=N.H$ for an extension of $N$ by $H$; in other words $G$ contains a normal subgroup $N$ such that $G/N\cong H$. We write $G=N:H$ if the extension is split.

\begin{lemma}\label{L: oddnormal}
Let $H$ be a group containing an involution $g$ and let $N\lhd H$. If $|N|$ is odd then
$$|H:C_H(g)|=|N:C_N(g)|\times|H/N: C_{H/N}(gN)|.$$
\end{lemma}
\begin{proof}
Take $C\leq H$ such that $C/N= C_{H/N}(gN)$. Then $C\geq C_H(g)$. Let $N^*=\langle g,N \rangle\cong N.2$ and take $c\in C$. Then $g^c\in N^*$.

Since $|N|$ is odd this implies that $g^{cn}=g$ for some $n\in N$ by Sylow's theorem. Thus $C=N.C_H(g)$. Then
$$|H:C|=|H:N.C_H(g)|=\frac{|H:C_H(g)|}{|N:N\cap C_H(g)|}=\frac{|H:C_H(g)|}{|N:C_N(g)|}.$$
Since $|H:C|=|H/N: C_{H/N}(gN)|$ we are done.
\end{proof}

\begin{lemma}\label{L: sylowtwos}
Let $H$ be a group and let $N\lhd H$. Suppose that $g$ is an involution contained in $N$. Let $P$ be a Sylow $2$-subgroup of $N$. Then
$$|H:C_H(g)|=|N:C_N(g)|\frac{|g^H\cap P|}{|g^N\cap P|}.$$
\end{lemma}
\begin{proof}
Observe that, by the Frattini argument, the set of $N$-conjugates of $P$ is the same as the set of $H$-conjugates of $P$; say there are $c$ of these. Let $d$ be the number of such $N$-conjugates of $P$ which contain the element $g$. Now count the size of the following set in two different ways:
$$|\{(x,Q):x\in g^H, Q\in P^H, x\in Q\}| = |H:C_H(g)|d = c|g^H\cap P|.$$
Similarly we count the size of the following set in two different ways:
$$|\{(x,Q):x\in g^N, Q\in P^N, x\in Q\}| = |N:C_N(g)|d = c|g^N\cap P|.$$
Our result follows.
\end{proof}

Combining these two results we have the following:

\begin{lemma}\label{L: invcentralizer}
Let $H$ be a subgroup of $H_1\times\dots\times H_r$ containing an involution $g$. For $i=1,\dots, r$, let $L_i$ be the projection of $H$ to $H_i\times H_{i+1}\times\dots\times H_r$. For $i=1,\dots, r-1$ let $\psi_i:L_i\to L_{i+1}$ be the canonical projection and let $T_i$ for the kernel of $\psi_i$. Define $T_r$ to be equal to $L_r$. Finally let $g_i$ be the image of $g$ under the projection into $L_i$.

Suppose that $T_i$ has odd order for $i<k\leq r$ and $T_k$ has even order. Let $P$ be a Sylow $2$-subgroup of $T_k$. Then
$$|H:C_H(g)|=\left(\prod_{i=1}^{k}|T_i:C_{T_i}(g_i)|\right)\frac{|(g_k)^{L_k}\cap P|}{|(g_k)^{T_k}\cap P|}.$$
\end{lemma}
\begin{proof}
If $T_1$ has odd order then Lemma \ref{L: oddnormal} implies that
\begin{eqnarray*}
|H:C_H(g)|&=&|T_1:C_{T_1}(g)|\times |H/T_1: C_{H/T_1}(T_1g)| \\
&=& |T_1:C_{T_1}(g)|\times |\psi_1(H): C_{\psi_1(H)}(\psi_1(g))| \\
&=& |T_1:C_{T_1}(g_1)|\times |L_2: C_{L_2}(g_2)|.
\end{eqnarray*}
Now $L_2$ is a subgroup of $H_2\times\dots\times H_r$ and so we can iterate the procedure. This implies that
$$|H:C_H(g)|=\left(\prod_{i=1}^{k-1}|T_i:C_{T_i}(g_i)|\right)|L_k:C_{L_k}(g_k)|.$$
If $k=r$ then we are done. If $k<r$ then we must calculate the centralizer of $g_{k}$ in $L_{k}\leq H_k\times\dots\times H_r$. Then we apply Lemma \ref{L: sylowtwos} using $T_k$ for our normal subgroup $N$. Then
$$|L_k:C_{L_k}(g_k)|=|T_k:C_{T_k}(g_k)|\times \frac{|(g_k)^{L_k}\cap P|}{|(g_k)^{T_k}\cap P|}.$$
\end{proof}

We conclude this subsection with a result concerning Sylow $2$-subgroups of $GL_n(q)$. 

\begin{lemma}\label{L: sylowtwoingln}
Let $P$ be a $2$-group in $GL_n(q)$ where $q=p^a$, $p$ is prime and $p\geq 7, p\equiv 1(3)$.
\begin{enumerate}
\item If $q\equiv 3(4)$, $q>7$, $n>2$ and $(q,n)\neq (31,4)$ then
$$|P|<q^{n-1}+\dots+q+1.$$
\item If $(q,n)=(31,4)$ then $P$ contains less than $31^3+31^2+31+1$ involutions.
\item If  $q\equiv 3(4)$ and $n=2$ then $P$ contains at most $q+2$ involutions. Of these at most $q+1$ are non-central in $GL_n(q)$.
\item If $q=7$ and $n>2$ then $P$ contains less than $q^{n-1}+\dots+q+1$ involutions.
\item If $q\equiv 1(4)$ then $P$ contains less than $q^{n-1}+\dots+q+1$ involutions.

\end{enumerate}
\end{lemma}
\begin{proof}
It is sufficient to prove these results for the case where $P$ is a Sylow $2$-subgroup of $GL_n(q)$; so assume this. Furthermore, throughout this proof we use the following notation: For an integer $k$, we write $k_2$ to mean the highest power of $2$ which divides $k$.

{\bf Statement 1}: Suppose first of all that $q\equiv 3(4)$. Then
\begin{eqnarray*}
|GL_n(q)|_2 &=& \prod_{i=1}^n|q^i-1|_2=|q-1|_2^n\times |q+1|_2^{\lfloor\frac{n}2\rfloor}\times\prod_{i=1}^{\lfloor\frac{n}2\rfloor}\frac{q^{2i}-1}{q^2-1} \\
&=&2^n\times|q+1|_2^{\lfloor\frac{n}2\rfloor}\times\prod_{i=1}^{\lfloor\frac{n}2\rfloor}i_2\leq 2^n\times|q+1|_2^{\lfloor\frac{n}2\rfloor}\times 2^{\frac{n}2-1} \\
&=&2^{\frac{3n}2-1}\times|q+1|_2^{\lfloor\frac{n}2\rfloor}\leq\frac12(8|q+1|_2)^{\frac{n}{2}}.
\end{eqnarray*}
Now, if $q\geq 31$ and $n\geq 6$ then $\frac12(8(q+1))^{\frac{n}{2}}<q^{n-1}$. If $q=19$ then $|q+1|_2=4$ and $\frac12(8|q+1|_2)^{\frac{n}{2}}<19^{n-1}$ for $n\geq 3$. Hence to prove the first statement we need only examine the situation when $(q,n)=(31,3)$ or $(31,5)$ - this can be done directly.

{\bf Statement 2}: Observe that a Sylow $2$-subgroup of $GL_4(31)$ is isomorphic to $P_2\wr 2$ where $P_2$ is a Sylow $2$-subgroup of $GL_2(31)$. It is easy to see that $P_2$ contains $33$ involutions  (see the proof of Statement 3) and we deduce that $P=P_2\wr 2$ must contain less than $2\times 34^2$ involutions; this bound is sufficient.

{\bf Statement 3}: In the third case $P$ lies inside a group, $H,$ of order $4(q+1)$ isomorphic to 
$$\langle a,b: a^{2(q+1)}=b^4=1, a^{q+1}=b^2, b^{-1}ab=b^q\rangle.$$
All elements of $H$ can be written as $a^i$ or $a^ib$ where $i=1,\dots, (2q+2)$. Then $H$ contains at most $q+1$ involutions of form $a^ib$ (corresponding to the situation when $i$ is odd). Furthermore $a^{q+1}$ is an involution and it is central in $GL_n(q)$. The third statement follows. 

{\bf Statement 4}: Suppose next that $P$ is a Sylow $2$-group in $GL_n(7)$. Suppose first that $n$ is even, write $n=2k$. Then $P$ is isomorphic to $P_2\wr N_2$ where $P_2$ is a Sylow $2$-subgroup in $GL_2(7)$ and $N_2$ is a Sylow $2$-subgroup in $S_{k}$, the symmetric group on $k$ letters. 

Observe first of all that $P_2$ has size $32$ and contains $9$ involutions (see the proof of Statement 3). Now write $P=M:N_2$ where $M=\underbrace{P_2\times\cdots\times P_2}_{k}$. Then $N_2$ can be thought of as acting on $M$ by permuting components; thus, for elements $h\in N_2, (m_1,\dots, m_k)\in M$, we have
$$(m_1,\dots, m_k)^h = (m_{h(1)},\dots,m_{h(k)}).$$

Now consider an element $g=(m_1,\dots, m_k).h\in M:N_2$ such that $g^2=1$. Clearly we must have $h^2=1$. Thus
$$g^2=(m_1m_{h(1)}, m_2m_{h(2)}, \dots,  m_{k}m_{h(k)}).$$
Now consider the action of $h$ on $\{1,\dots, k\}$. Reorder so that $h$ fixes the values $1,\dots,f$ and has orbits of order $2$ on the remaining $k-f$ elements. We count the number of values which the $m_i$ can take: Clearly $m_1,\dots, m_f$ must have order at most $2$ in $P_2$ so there are $10$ possible choices for each of these. Furthermore we must have $m_{l(h)}=m_l^{-1}$ for $l>f$. Thus for a particular value of $h$, there are $10^f\times 32^{\frac{k-f}2}$ values for $(m_1,\dots, m_k)$ such that $g=(m_1,\dots, m_k).h$ satisfies $g^2=1$. In particular there are at most $10^k$ such values.

Thus, for every involution (plus the identity) in $N_2$ there are less than $10^{\frac{n}2}$ involutions in $P$. Hence the number of involutions in $P$ is less than
$$10^{\frac{n}2}|N_2|\leq 10^{\frac{n}2}2^{\frac{n}2-1}=\frac12 20^{\frac{n}2}.$$
This satisfies the bound of the fourth statement.
If $n$ is odd then $P=C_2\times P_{n-1}$ where $C_2$ is the group of order $2$ and $P_{n-1}$ is a Sylow $2$-subgroup in $GL_{n-1}(7)$. Since $n-1$ is even we know that the number of involutions in $P_{n-1}$ is less than $q^{n-2}+\cdots+q+1$. Then the number of involutions in $P$ is less than
$$2\times (q^{n-2}+\cdots+q+1) +1.$$
This is less than $q^{n-1}+\cdots+q+1$ as required.

{\bf Statement 5}: Suppose next that $q\equiv 1(4)$ and let $P$ be a Sylow $2$-subgroup of $GL_n(q)$. Then $P\leq C\wr N_2$ where $C$ is a cyclic subgroup of order dividing $\frac{q-1}3$ and $N_2$ is a Sylow $2$-subgroup of $S_n$. We can proceed similarly to the fourth case. Write $M=\underbrace{C\times\cdots\times C}_n$ and consider an involution $g=(m_1,\dots, m_n).h.$ As before $h^2=1$. Furthermore for a particular choice of $h$ there are at most $2^f \times (\frac{q-1}{3})^{\frac{n-f}2}$ possible values for $(m_1, \dots, m_n)$ such that $g^2=1$. In particular there are at most $(\frac{q-1}{3})^{\frac{n}2}$ such values.

Now $|N_2|\leq 2^{n-1}$ and so $P$ has less than $(\frac{q-1}3)^{\frac{n}2}2^{n-1}$ involutions. This implies Statement 5.
\end{proof}

\subsection{The projective plane situation}\label{SS: pp}

This subsection is the last, until Section \ref{S: quaternionsylowtwos}, in which we will directly consider projective planes. Hence all the notation in this subsection is self-contained and we will develop this notation as we go along.

We begin by stating a hypothesis which will hold throughout this subsection. The conditions included represent, by \cite{wagner} and \cite[4.1.7]{dembov}, the conditions under which a group may act transitively on the points of a non-Desarguesian projective plane.

\begin{hypothesis}\label{h:basic}
Suppose that a group $G$ acts transitively upon the points of a non-Desarguesian projective plane $\spaceP$ of order $x>4$. If $G$ contains any involutions then they fix a Baer subplane; in particular they fix $u^2+u+1$ points where $x=u^2, u>2$. Furthermore in this case any element of $G$ fixing at least $u^2$ points fixes either $u^2+u+1, u^2+1$ or $u^2+2$ points of $\spaceP$.
\end{hypothesis}

We collect some significant facts that follow from Hypothesis \ref{h:basic}:

\begin{lemma}\label{L:firstly}
The Fitting group and the generalized Fitting group of $G$ coincide, i.e. $F^*(G)=F(G).$ What is more, $F(G)$ acts semi-regularly on the points of $\spaceP$. Further, if $p$ is a prime dividing into $u^2+u+1,$ then $p\equiv 1(3)$ or $p=3$ and $9$ does not divide into $u^2+u+1$.
\end{lemma}
\begin{proof}
The results about the Fitting group of $G$ can be found in \cite[Theorem A]{gill2} and \cite[Theorem A]{gill4}. The result about prime divisors can be found in \cite[p.33]{kantor}.
\end{proof}

Note that $u^2-u+1=(u-1)^2+(u-1)+1$ hence the statement about prime divisors also holds for $u^2-u+1$.

Write $\alpha$ for a point of $\spaceP$. For an integers $k$ and $w$, write  $k_w$ (resp. $k_{w'}$) for the largest divisor of $k$ which is a power of $w$ (resp. coprime to $w$). For an element $g\in G$ we write $g^G$ for the set of $G$-conjugates of $g$ in $G$; we also write $\Fix(g)$ for the set of fixed points of $g$. Similarly $\Fix(H)$ is the set of fixed points of a subgroup $H<G$. We begin with the following observation:

\begin{lemma}\label{L: threepossibilities}
Suppose that $G$ contains an involution. Then one of the following holds:
\begin{enumerate}
\item $m_2(G)=1$;

\item All primes which divide $|F(G)|$ must also divide $u^2+u+1$.
\end{enumerate}
\end{lemma}
\begin{proof}
Suppose that $m_2(G)>1$. Let $N$ be a Sylow $r$-subgroup of $F(G)$, for some prime $r$, and observe that any subgroup of $G$ acts by conjugation on $N$. Then \cite[(40.6)]{aschbacher3} implies that a Sylow $2$-subgroup of $G$ does not act semi-regularly on $N$. Hence $G$ contains an involution $g$ for which $C_N(g)$ is non-trivial. Now $C_N(g)$ acts on $\Fix(g)$, a set of size $u^2+u+1$. Because $F(G)$ acts semi-regularly on the points of $\spaceP$ we conclude that $r$ divides $u^2+u+1$.
\end{proof}

We will be interested in the second of these possibilities. So for the rest of this subsection we add the following to our hypothesis:
\begin{hypothesis}\label{h:main}
Suppose that $G$ contains an involution and that all primes dividing $|F(G)|$ also divide $u^2+u+1$. 
\end{hypothesis}

Clearly if we can show that Hypotheses 1 and 2 lead to a contradiction then Lemma \ref{L: threepossibilities} will imply Theorem A. Over the rest of this subsection we will work towards showing that, provided Lemma A is true, such a contradiction does indeed follow from these hypotheses.

Write $u^2+u+1=p_1^{a_1}\cdots p_r^{a_r}$ and observe that, by Lemma \ref{L:firstly}, $p_i\equiv 1(3)$ or else $p_i^{a_i}=3$. Let $F(G)=N_1\times N_2\times\dots\times N_r$ where $N_i\in Syl_{p_i} F(G)$ and set $Z=Z(F(G))$. Then $G/Z$ is a subgroup of 
$$Aut N_1\times\dots\times Aut N_r.$$
Write $V_i=N_i/\Phi(N_i)$. Here $\Phi(N_i)$ is the Frattini subgroup of $N_i$, hence $V_i$ is a vector space over the field of size $p_i$. Observe that $|V_i|\leq p_i^{a_i}$.

A classical result of Burnside (see, for instance, \cite[Theorem 1.4]{gorenstein3}) tells us that $Aut N_i$ acts on $V_i$ with kernel, $K_i$, a $p_i$-group. Thus $G/Z$ is a subgroup of
$$K_1.GL(V_1)\times\dots\times K_r.GL(V_r).$$
Let $K^\dagger =(K_1\times\dots\times K_r) \cap G/Z$ and take $K$ to be the pre-image of $K^\dagger$ in $G$. Thus $G/K$ is a subgroup of
$$GL(V_1)\times\dots\times GL(V_r).$$

\begin{lemma}\cite[Lemma 13]{gill2}\label{L: counting}
Let $x=u^2$ and let $g$ be an involution in $G$. Then
$$\frac{|g^G|}{|g^G\cap\Galph|}=u^2-u+1.$$
\end{lemma}

\begin{lemma}\label{L: divide}
Let $g$ be an involution in $G$. Then $u^2-u+1$ divides
$$|G/K:C_{G/K}(gK)|.$$
\end{lemma}
\begin{proof}
By Lemma \ref{L: counting}, $u^2-u+1$ divides $|G:C_G(g)|$. Then, by Lemma \ref{L: oddnormal},
$$|G:C_G(g)|=|K:C_K(g)|\times|G/K:C_{G/K}(gK)|.$$
Now all primes dividing $|K|$ also divide $u^2+u+1$. But $u^2+u+1$ is coprime to $u^2-u+1$ so we have our result.
\end{proof}

We wish to apply Lemma \ref{L: invcentralizer} to the group $G/K$ which is a subgroup of $GL(V_1)\times\dots\times GL(V_r).$ Recall that $u^2+u+1=p_1^{a_1}\cdots p_r^{a_r}$ where the $p_i$ are prime numbers; furthermore $V_i=p_i^{b_i}$ for some $b_i\leq a_i$.

\begin{lemma}\label{L: centralizerbound}
Define $L_i$, $\psi_i$ and $T_i$ for $G/K$ just as in Lemma \ref{L: invcentralizer}. We choose an ordering so that there exists $k$ with $|T_i|$ odd for $i<k$ and $|T_i|$ even for $i\geq k$. Then
 $$(|T_j:C_{T_j}(g_j)|, u^2-u+1)>p_j^{a_j-1}+\dots+p_j+1$$ for some $j\leq k$.
\end{lemma}
\begin{proof}
We suppose that the proposition does not hold and seek a contradiction. Lemmas \ref{L: invcentralizer} and \ref{L: divide} imply that $u^2-u+1$ divides
$$\left(\prod_{i=1}^{k}|T_i:C_{T_i}(g_i)|\right)\frac{|(g_k)^{L_k}\cap P|}{|(g)^{T_k}\cap P|}$$
where $P$ is a Sylow $2$-subgroup of $T_k$.

If $k=r$ then $T_r=L_r$ and so
$$u^2-u+1\leq\prod_{i=1}^r(p_i^{a_i-1}+\dots+p_i+1)<\frac{u^2+u+1}{2^r}.$$
This is clearly a contradiction.

Now suppose that $k<r$ and we choose an ordering for the $p_i$ such that $p_i^{a_i}<p_j^{a_j}$ whenever $i<j$. Then Lemma \ref{L: sylowtwoingln} implies that
$$u^2-u+1\leq\left(\prod_{i=1}^{k}(p_i^{a_i-1}+\dots+p_i+1)\right)(p_k^{a_k-1}+\dots+p_k+1).$$
Then, by virtue of our ordering, we have
$$u^2-u+1\leq\prod_{i=1}^{k+1}(p_i^{a_i-1}+\dots+p_i+1)<\frac{u^2+u+1}{2^{k+1}}.$$
Once again we have a contradiction.
\end{proof}

We are now in a position to state Lemma A. First of all we make a definition:

\begin{definition}\label{D: heart} 
For an integer $k$ we write $k_\heartsuit$ to mean
$$(k,3)\times k_7\times k_{13} \times k_{19}\times  k_{31} \times \dots$$
(We must consider all primes equivalent to $1$ modulo $3$ here.)
\end{definition}

\begin{lemmaa}
Let $H<GL_n(q)$ and suppose that $H$ has even order. Suppose that $q=p^a$ with $p\geq 7$ and $p\equiv 1(3)$. Then there exists an involution $g\in H$ such that
$$|H:C_H(g)|_{p',\heartsuit}\leq q^{n-1}+\dots+q+1.$$
\end{lemmaa}

Note that, for an integer $k$, by $k_{p',\heartsuit}$ we mean the largest integer dividing $k_\heartsuit$ which is coprime to $p$.

\begin{proposition}\label{prop:implication}
Lemma A implies Theorem A.
\end{proposition}
\begin{proof}
Consider $G/K$ as a subgroup of $GL(V_1)\times \dots\times GL(V_r)$. Let $L_i$ be the projection of $G/K$ onto $GL(V_i)\times\dots\times GL(V_r)$. Let $T_i$ be the kernel of the projection $L_i\to L_{i+1}$ for $i=1,\dots, r-1$ and let $T_r=L_r$.

Order the $V_i$ so that $T_i$ has odd order for $i<k\leq r$ and $T_k$ has even order. Recall that $u^2+u+1=p_1^{a_1}\dots p_r^{a_r}$ where, by Lemma \ref{L:firstly}, either $p_i\geq 7$ and $p_i\equiv 1(3)$ or $p_i^{a_i}=3$. Then $|V_i|=p_i^{b_i}$ where $b_i\leq a_i$.

We first apply Lemma A to $T_k$ which can be thought of as a subgroup of $GL(V_k)$. Thus let $g_k$ be an involution in $T_k$ such that $|T_k:C_{T_k}(g_k)|_{p_k',\heartsuit}$ is minimised. Lemma A implies that
$$|T_k:C_{T_k}(g_k)|_{p_k',\heartsuit}\leq p_k^{b_k-1}+\dots+p_k+1\leq p_k^{a_k-1}+\dots+p_k+1.$$
Now, by Lemma \ref{L:firstly} (and the comment immediately after), this implies that
$$(|T_k:C_{T_k}(g_k)|, u^2-u+1)\leq |T_k:C_{T_k}(g_k)|_{p_k',\heartsuit}\leq p_k^{a_k-1}+\dots+p_k+1.$$

Now we adjust our notation slightly, as follows. Let $g$ be a pre-image of $g_k$ in $G$ and write $g=(g_1,\dots, g_k, 1,\dots, 1)$. In our old notation $g_i$ was an element in $T_i$ and hence had form $(h_i, 1,\dots, 1)$ for some $h_i$ in $GL(V_i)$. In our new notation we simply identify $g_i$ with $h_i$ and observe that the statements of Lemma \ref{L: invcentralizer} and \ref{L: centralizerbound} hold under this new notation.

Now, when $j<k$, $T_j$ has odd order and is isomorphic to a subgroup of $GL(V_j)$. Then $H_j=\langle T_j, g_j\rangle$ has a unique $H_j$-conjugacy class of involutions. Thus Lemma A implies that, for all $j<k$,
$$(|T_j:C_{T_j}(g_j)|, u^2-u+1)\leq p_j^{a_j-1}+\dots+p_j+1.$$
This yields a contradiction to Lemma \ref{L: centralizerbound}. 

Thus we have demonstrated that, provided Lemma A is true, our Hypotheses 1 and 2 lead to a contradiction. Thus, by Lemma \ref{L: threepossibilities}, Lemma A implies Theorem A.
\end{proof}

\section{Proving Lemma A}\label{S: progress}

Throughout this section we occupy ourselves with a proof of Lemma A. Lemma A is a purely group theoretic result; we will not refer to projective planes in this section. We use standard group theory notation, as described at the start of Subsection \ref{S: background}.

Recall that, for integers $k$ and $w$, we write $k_w$ (resp. $k_{w'}$) for the largest divisor of $k$ which is a power of $w$ (resp. coprime to $w$). Furthermore we write $k_{\heartsuit}$ for $$(k,3)\times k_7\times k_{13} \times k_{19}\times  k_{31} \times \dots$$
(We must consider all primes equivalent to $1$ modulo $3$ here.)

Then the result we have to prove is the following:

\begin{lemmaa}
Let $H<GL_n(q)$ and suppose that $H$ has even order. Suppose that $q=p^a$ with $p\geq 7$ and $p\equiv 1(3)$. Then there exists an involution $g\in H$ such that
$$|H:C_H(g)|_{p',\heartsuit}\leq q^{n-1}+\dots+q+1.$$
\end{lemmaa}

Throughout this section $q=p^a$ with $p\geq 7$ and $p\equiv 1(3)$ and $H$ is a subgroup of $GL_n(q)$. Suppose that $H$ lies in a maximal subgroup $M$ of $GL_n(q)$. We use the result of Aschbacher\cite{aschbacher2} as described in \cite{kl}. Either $H$ lies in the class $\mathcal{S}$ or $M$ lies in one of eight families. We will say that $M$ {\it is of type $i$} if it lies inside family $i$ as described in \cite{kl}.

In order to prove Lemma A we will go through the possibilities for $M$ and $H$ and demonstrate that, in all cases, Lemma A holds. We start with a series of lemmas which demonstrate that Lemma A holds if $H$ is in class $\mathcal{S}$.

\begin{lemma}\label{L: simplepgl}
Let $H_1$ be a non-abelian simple subgroup of $PGL_n(q)$. Then $H_1$ contains an involution $g$ such that
$|AutH_1:C_{Aut H_1}(g)|_{p',\heartsuit}\leq q^{n-1}+\dots+q+1$
unless $H_1=PSL_m(p^b)$ with $m$ even.
\end{lemma}
\begin{proof}
The lemma is proved using information in \cite{kl}. In particular we use Propositions 5.3.7, 5.3.8 and 5.4.13 as well as Theorem 5.3.9. These results give a minimum value for $n$ depending on the isomorphism class of $H_1$. We will not go through all possibilities for $H_1$ here; instead we give several sample calculations to demonstrate our method.

Suppose that $H_1=A_m$ is the alternating group on $m$ with $m\geq 9$. Then $Aut H_1\cong S_m$. Now $H_1$ contains an involution $g$ (a double transposition) such that 
$$|S_m: C_{S_m}(g)| = \frac{m(m-1)(m-2)(m-3)}{8}.$$
In particular $|S_m: C_{S_m}(g)|<7^{m-3}$ for $m\geq 9$. Now \cite[Proposition 5.3.7]{kl} implies that $n\geq m-2$ and so $|S_m: C_{S_m}(g)|<7^{m-3}\leq q^{n-1}$ as required.

For $m=5,6,7$ and $8$ the result can be proved easily using \cite[Proposition 5.3.7]{kl}.

If $H_1$ is a sporadic group then \cite[Proposition 5.3.8]{kl} gives a minimum value for $n$. Furthermore \cite[Table 5.1.C]{kl} asserts that the outer automorphism group of $H_1$ divides $2$. Hence, since $2\not\equiv 1(3)$, we need only prove that
$|H_1:C_{H_1}(g)|_{p',\heartsuit}\leq q^{n-1}+\dots+q+1.$

In fact for the sporadic groups we are able to prove that $|H_1|_\heartsuit\leq 7^{n-1}$ which is sufficient. For instance if $H_1$ is one of the Mathieu groups then $|H_1|_\heartsuit$ divides $3\times 7$. Since $n\geq 5$ for the Mathieu groups the result follows.
 
If $H_1$ is a group of Lie type of characteristic coprime to $p$ then \cite[Theorem 5.3.9]{kl} gives a minimum value for $n$. We present two example calculations: If $H_1=E_6(r)$ then $n>{r^9(r^2-1)}$. Now observe that $|H_1|<r^{78}$. Since $q>7$ it is sufficient to observe that $r^{78}<7^{r^9(r^2-1)-1}\leq q^{n-1}$. If $H_1=\Omega_{2m+1}(r)$ with $r$ odd and $(r,m)\neq(3,3)$ then $n>r^{2(m-1)}-r^{m-1}$. Now $|H_1|<r^{2m^2+m}$ thus it is sufficient to observe that $r^{2m^2+m}<7^{r^{2(m-1)}-r^{m-1}-1}\leq q^{n-1}$ for $m\geq 3$. The exception can be dealt with easily. 

Finally if $H_1$ is a group of Lie type of characteristic $p$ then , since $p\geq 7,$ we need not consider ${^2G_2(q)}$, ${^2F_4'(q)}$ or ${^2B_2(q)}$. We need to split into two cases. Suppose that $H_1\cong X_m(q_0)$ (i.e. has dimension $m$ and is defined over a field of size $q_0$). Suppose first that $q_0\leq q$; then it is sufficient to prove the lemma for $q_0=q$. Suppose this is the case and refer to \cite[Proposition 5.4.13]{kl} which gives a minimum value for $n$. 

We present two examples: If $H_1\cong G_2(q)$ then $n\geq7$. But $H_1$ contains an involution $g$ such that $|Aut H_1:C_{AutH_1}(g)|_{p'}=q^4+q^2+1<q^{n-1}$. Similarly if $H_1=PSp_{m}(q)$ with $m>4$ then $n\geq m$. But $H_1$ contains an involution $g$ such that $|Aut H_1:C_{AutH_1}(g)|_{p'}=q^{m-2}+\dots+q^2+1<q^{n-1}$.

Our argument here does not work for $H_1\cong PSL_n(q)$ when $n$ is even. Then $PSL_n(q)\leq PGL_n(q)$ and, for all involutions $g\in H_1$, 
$$|H_1: C_{H_1}(g)|_{p'}\geq (q^{n-2}+\dots+q+1)(q^{n-2}+\dots+q^2+1).$$

On the other hand suppose that $q<q_0$. By considering the order of $H_1\cong X_m(q_0)$ it is possible to bound $q^n$ by some function of $q_0$ and $m$. We present two examples: Suppose that $H_1\cong PSL_m(q_0)$. Then $q^n\geq q_0^m$ and, in particular, $q^{n-1}+\dots+q+1\geq q_0^{m-1}+\dots+q+1$. Now, for $m$ odd, we have an involution in $H_1$ such that $|Aut H_1:C_{AutH_1}(g)|_{p'}=q_0^{m-1}+\dots+q+1$ which satisfies the required bound.

Similarly if $H_1\cong PSU_m(q_0)$ with $m$ odd. Then, by considering orders, we must have $q^n>q_0^{2m}$. In particular this means that
$$q^{n-1}+\dots+q+1\geq q_0^{2m-1}+\dots+q_0+1.$$
Now there exists an involution in $H_1$ such that $|Aut H_1:C_{AutH_1}(g)|_{p'}=q_0^{m-1}-\dots-q_0+1$ which satisfies the required bound.

This argument can be followed through in all cases except when $H_1$ is isomorphic to one of the following groups: $P\Omega^+_m(q_0)$, $E_6(q_0)$, $E_7(q_0)$ and $E_8(q_0)$. In these cases we note that \cite[Proposition 5.4.13]{kl} implies that $n\geq 8$. Now we use the main theorem from \cite{liebeck3}. This gives two possibilities as follows.

The first possibility is that $|H_1|<q^{3n}$. If $H_1\cong E_6(q_0)$ then, since $|E_6(q^0)|>q_0^{72}$, we have $q_0^{24}<q^n$. Now $H_1$ contains an involution such that $|Aut H_1:C_{AutH_1}(g)|_{p'}=(q_0^6+q_0^3+1)(q_0^2+q_0+1)(q_0^8+q_0^4+1)<\frac{q^{n}}{q_0}<q^{n-1}$ as required. The other cases are similar.

The second possibility is that $H_1$ lies inside a maximal subgroup $M$ of type $i$, for $i=1,\dots, 8$. If $i\neq 3,6$ or $8$, this implies that $H_1$ is a subgroup of $PGL_m(q_1)$ where $m\leq n$ and $q_1\leq n$. It is then sufficient to prove the bound within $PGL_m(q_1)$, so we iterate our analysis. If $i=3$ then $H_1$ is a subgroup of $PGL_m(q_1)$ where $m<n$ and $q_1^m=q^n$. Once again it is sufficient to prove the bound within $PGL_m(q_1)$. If $i=6$ then $M$ is so small that $H_1$ must satisfy the bound (see Lemma \ref{L: six} below). 

Finally we suppose that the only maximal subgroups which contain $H_1$ are of type $8$. In fact, since $n\geq 8$ and $q<q_0$, this is impossible. Hence the bound is satisfied in all cases.
\end{proof}

\begin{corollary}\label{cor:simplepgl}
Let $H_1$ be a quasi-simple subgroup of $PGL_n(q)$. Then $H_1$ contains an involution $g$ such that
$|AutH_1:C_{Aut H_1}(g)|_{p',\heartsuit}\leq q^{n-1}+\dots+q+1$
unless $H_1/Z(H_1)=PSL_m(p^b)$ with $m$ even.
\end{corollary}
\begin{proof}
Write $Z=Z(H_1)$. Observe first that 
$$|AutH_1:C_{Aut H_1}(g)|_\heartsuit=|Aut(H_1/Z): C_{Aut(H_1/Z)}(gZ)|_\heartsuit.$$
Next we refer to \cite[Corollary 5.3.3]{kl} which states, for $k$ algebraically closed, if $H_1<PGL_n(k)$ then $H_1/Z<PGL_m(k)$ for some $m\leq n$.

In our proof of Lemma \ref{L: simplepgl} we have only considered the size of $q$ when $H_1$ was a group of Lie type of characteristic $p$. Thus, except in this case, we have shown that
$$|AutH_1:C_{Aut H_1}(g)|_\heartsuit=|Aut(H_1/Z): C_{Aut(H_1/Z)}(gZ)|_\heartsuit\leq q^{n-1}+\dots+q+1.$$

Now assume that $H_1\cong X_m(q_0)$ is a quasi-simple group of Lie type of characteristic $p\geq 7$. We assume that $H_1$ is not simple. For the situation when $q_0\leq q$ the argument in Lemma \ref{L: simplepgl} transfers directly hence we assume that $q<q_0$. Again, as in Lemma \ref{L: simplepgl}, by comparing orders we are able to bound $q^n$ in terms of $q_0$ and $m$ and so can establish the required inequality in most cases.

The remaining cases are for $H_1$ a covering group of $P\Omega^+_m(q_0)$ or a simply connected group of type $E_6(q_0)$ or $E_7(q_0)$. If $H_1$ has a central involution then the result clearly stands hence we only need to deal with the case when $H_1$ is a simply connected group of type $E_6 (q_0)$. But in this case we simply apply the main theorem of \cite{liebeck3} as described in the proof of Lemma \ref{L: simplepgl}.
\end{proof}

\begin{proposition}\label{prop:simplegl}
Let $H$ be a quasi-simple subgroup of $GL_n(q)$. If $H/Z(H)$ is isomorphic to $PSL_m(p^b)$ with $m$ even then assume that $H$ is absolutely irreducible in $GL_n(q)$. Then $H$ contains an involution $g$ such that
$|AutH:C_{Aut H}(g)|_{p',\heartsuit}\leq q^{n-1}+\dots+q+1$.
\end{proposition}
\begin{proof}
Observe that $H/(Z(GL_n(q)\cap H)<PGL_n(q)$. Then Corollary \ref{cor:simplepgl} implies the result unless $H/Z(H)=PSL_m(p^b)$ with $m$ even. Consider this exception and assume that $H$ is absolutely irreducible. We assume that $H$ does not contain a central involution. Now $H$ contains an involution $g$ with $|AutH: C_{AutH}(g)|_\heartsuit\leq (q^{m-2}+\dots+q+1)(q^{m-2}+\dots+q^2+1).$ Hence we need to show that
$$(q^{m-2}+\dots+q+1)(q^{m-2}+\dots+q^2+1)<q^{n-1}+\dots+q+1.$$
Thus it is enough to show that $n\geq 2m-2$.

Since $H$ is absolutely irreducible, Schur's Lemma implies that $H/Z(H)$ embeds into $PGL_n(k)$ (here $k$ is an algebraically closed field of characteristic $p$). Since $H$ does not contain a central involution this module must be different from the natural projective module of dimension $m$. If $n>\frac12 m(m+1)$ then $n\geq 2m-2$ and we are done. If $n\leq\frac12 m(m+1)$ then the dimensions of all projective $k$-modules of $H/Z(H)$ are listed in \cite[Table 5.4.A]{kl}. Then we have $n\geq \frac12 m(m-1)\geq 2m-2$ provided $m\geq 4$. 

If $m=2$ then $H=PSL_2(q)$ has a single conjugacy class of involutions of size $\frac 12 q(q\pm1)$ and we are done.

\end{proof}

\begin{corollary}\label{C: c8}
Suppose that all involutions in $H$ satisfy
$|H:C_H(g)|_{p',\heartsuit}>q^{n-1}+\dots+q+1$. Then $H$ lies in a maximal subgroup of $GL_n(q)$ of type 1 to 7.
\end{corollary}
\begin{proof}
If $H\geq SL_n(q)$ then $H$ either has a central involution or an involution $g$ such that $|H:C_H(g)|=q^{n-1}+\dots+q+1$. Thus this possibility can be excluded. Furthermore it is clear that $H$ must not lie in $\mathcal{S}$.

Now suppose that $H$ lies only in maximal subgroups of type 8. If $n\geq 7$ then this implies that $H$ contains a normal quasi-simple classical subgroup of dimension $n$. Then the previous proposition gives us our result. 

If $n=2$ then there are no maximal subgroups of type 8. Now suppose that $3\leq n\leq 6$ and $M$ is a maximal subgroup of type 8. If $n=3$ then we must consider the possibility that $(M\cap SL_3(q))/Z(GL_3(q))\cong PGL_2(q)$. But then $|M/Z(M)|_{p',\heartsuit}<q^2+q+1$ and we are done. 

For $n=4,5$ and $6$ it is easier to consider $H_1$ (resp. $M_1$), the section of $H$ (resp. $M$) in $PSL_n(q)$.

For $n=4$ we need to consider the possibility that $P\Omega_4^\epsilon(q)$ is involved in $M_1$. If $|H_1\cap P\Omega_4^\epsilon(q)|$ is even then $H_1$ contains an involution $g$ such that $|H_1:C_{H_1}(g)|_{p',\heartsuit}<(q+1)^2$. If, on the other hand $|H_1\cap P\Omega_4^\epsilon(q)|$ is odd then $|H_1|_{p',\heartsuit}<(q+1)^2$ as required.

For $n=5$ we need to consider the possibility that $\Omega_5(q)$ is involved in $M_1$. If $|H_1\cap \Omega_5(q)\cong PSp_4(q)|$ is even then $H_1$ contains an involution $g$ such that $|H_1:C_{H_1}(g)|_{p',\heartsuit}<q^2+1$. If, on the other hand $|H_1\cap\Omega_5(q)|$ is odd then $|H_1|_{p',\heartsuit}<q^3+q^2+q+1$ as required.

For $n=6$ we need to consider the possibility that $P\Omega_4^\epsilon(q)\cong PSL_4^\epsilon(q)$ is involved in $M_1$. If $|H_1\cap P\Omega_6^\epsilon(q)|$ is even then $H_1$ contains an involution $g$ such that $|H_1:C_{H_1}(g)|_{p',\heartsuit}<q^3+q^2+q+1$. If, on the other hand $|H_1\cap P\Omega_6^\epsilon(q)|$ is odd then $|H_1|_{p',\heartsuit}<q^4+q^3+q^2+q+1$ as required.
\end{proof}

This result means that we will need to examine subgroups lying in maximal subgroups of type 1 to 7. Some cases are easy to rule out as we note in the next two lemmas which are proved easily using information in \cite{kl}.

\begin{lemma}\label{L: six}
Suppose that $H$ lies in a maximal subgroup of type 6. If $n>2$ then $|H|_{p', \heartsuit}<q^{\frac{3n-2}4}.$ If $n=2$ and $H$ has even order then there exists an involution $g\in H$ such that $|H:C_H(g)|_{p',\heartsuit}<q^{\frac{3n-2}4}.$
\end{lemma}
\begin{proof}
Let $H_1$ be the section of $H$ in $PSL_n(q)$. If $n=2$ then $H_1\cong A_4$ or $S_4$. It is easy to see that $|H:C_H(g)|_{p', \heartsuit}\leq 3$ in all cases.

For $n>2$ we suppose that $H$ is a maximal subgroup of type 6. For $n\leq 10$ we use \cite[Propositions 4.6.5 and 4.6.6]{kl} to check that the result holds. For $n\geq 11$ it is sufficient to prove that $|H_1|_{p', \heartsuit}<q^{\frac{3n-10}4}.$

If $n=r^m$ and $r\equiv 2(3)$ then, by \cite[Propositions 4.6.5 and 4.6.6]{kl}, $|H_1|_{p', \heartsuit}<r^{m^2+m}$. Now, for $r^m\geq 11$ we have $r^{m^2+m}<7^{\frac{3n-10}4}$ as required. We proceed similarly for $r=3$ and $r\equiv 1(3)$.
\end{proof}

\begin{lemma}\label{L: seven}
Suppose that $H$ lies in a maximal subgroup $M$ of type 7, i.e. $H\leq GL_m(q)\wr S_t$ with $n=m^t$ and $m\geq 3$. If $t>2$ then $|H|_{p', \heartsuit}<q^{\frac{3n-2}4}.$ If $t=2$ and $H$ has even order then there exists an involution $g\in H$ such that $|H:C_H(g)|_{p',\heartsuit}<q^{\frac{3n-2}4}.$
\end{lemma}
\begin{proof}
Observe first that $t!<7^{t^2}$. In addition $|GL_n(q)|_{p'}<q^{\frac{m(m+1)}{2}}$ hence $|H|_{p'}<q^{\frac{tm(m+1)}{2}+t^2}$. It is therefore sufficient to prove that
$$\frac{tm(m+1)}{2}+t^2<\frac34m^t - \frac12.$$
It is easy to see that this inequality holds, provided $m>2$ and $(m,t)\neq (3,3)$. What's more if $(m,t)=(3,3)$ then $|H|_{p', \heartsuit}<(\prod_{i=1}^3(q^i-1))^3\times 3<q^{\frac{79}{4}}$ as required.

So suppose that $m=2$ and $M=(GL_m(q)\times GL_m(q)):2$. Define $N=GL_m(q)\times GL_m(q)$ to be the normal subgroup in $M$ of index $2$. Let $g$ be an involution in $M\backslash N$; then 
$$|M:C_M(g)|=|GL_m(q)|.$$
Hence if $H$ contains an involution $g$ outside $N\cap H$ then $|H:C_H(g)|_{p'}<q^{\frac{m(m+1)}2}$ which satisfies our bound provided $m\geq 3$.

If, on the other hand, $H$ contains no such involution then we take $g\in N$. The largest conjugacy class of involutions in $N$ satisfies
$$|N:C_N(g)|_{p'}=\left(\frac{\prod_{i=1}^m(q^i-1)}{\prod_{i=1}^{\lfloor\frac{m}2\rfloor}(q^i-1)\prod_{i=1}^{\lceil\frac{m}2\rceil}(q^i-1)}\right)^2.$$
Then $|H:C_H(g)|_{p'}\leq |N:C_N(g)|_{p'}<q^{\frac12m^2+\frac12m+\frac14}$ which satisfies our bound provided $m\geq 3$.

We are left with the possibility that $M=GL_2(q)\wr 2$ in $GL_4(q)$. Now if $g$ is an involution in $N$ then $N$ satisfies $|N:C_N(g)|_{p'}\leq \frac12(q+1)^2$. Thus if $H$ contains such an involution we are done. On the other hand if $|H\cap N|$ is odd then $H$ must contain an involution $g$ from $M\backslash N$, in which case $|H:C_H(g)|=|K|$ where $K$ is an odd order subgroup of $GL_2(q)$. This satisfies the required bound.
\end{proof}

Clearly, in seeking to prove our conjecture, we can assume that $H$ does not lie in a maximal subgroup of type 5 - since it would then be sufficient to prove the conjecture over the subfield. Thus we are left with the possibility that $H$ lies only in maximal subgroups from families 1 to 4.

Our method will be to proceed by induction on the dimension $n$. We give the base case, which is proved easily.

\begin{lemma}\label{L: base case}
Suppose that $n=1$ or $2$. Then $H$ contains an involution such that $|H:C_H(g)|_{p',\heartsuit}\leq q+1$.
\end{lemma}
\begin{proof}
If $M$ is of type 1 to 4 in $GL_2(q)$ then $M$ is isomorphic to $[q]:(q-1)^2$, $(q^2-1):2$ or $(q-1)^2:2$. Then $|M/(M\cap Z(GL_2(q)))|_{p'}\leq q+1$ as required. If $M$ is of type 5 then $M/(M\cap Z(GL_2(q)))\cong GL_2(q_0)$ for some $q_0<q$. Then it is sufficient to prove the bound over the smaller field. Finally if $M$ is not of these types then $|M/(M\cap Z(GL_2(q)))|_{p'}\leq 3$ as required.
\end{proof}

\subsection{Results about odd order subgroups}

In order to proceed with the inductive proof of Lemma A we will need a number of results concerning odd order subgroups. We state these results in this section, before returning to the proof of Lemma A in the next section. Recall that $q=p^a$ where $p\equiv 1(3)$ and $p\geq 7$.

\begin{lemma}\label{L: oddsn}
Let $H$ be a primitive subgroup of $S_n$, the symmetric group on $n$ letters. If $H$ has odd order then $|H|<n^{\log_2n}$.
\end{lemma}
\begin{proof}
If $n\leq 6$ then this is clearly true. Since $H$ has odd order, all minimal normal subgroups are elementary abelian. Let $P$ be such a minimal normal subgroup, $P=p_1^{b_1}$.

Referring to the O'Nan-Scott-Aschbacher theorem we see that $n=p_1^{b_1}$ and $H<P:GL_{b_1}(p_1)$. Since $p_1>2,$ $H$ has order less than $n^{\log_{p_1}n+1}<n^{\log_2n}$ for $n>6$.
\end{proof}

We now need a number of results about the odd order subgroups of $GL_n(q)$. Let $H$ be one such subgroup and observe that, for $n>1$, $H$ cannot contain $SL_n(q)$. Thus $H$ must lie inside a maximal subgroup of $GL_n(q)$ of type 1 to 8. The next three lemmas will all be proved using induction and by going through the possible maximal subgroups containing $H$ one type at a time.

In fact we can deal with one type straight away: Suppose that $n\geq 7$ and $H$ lies inside a maximal subgroup $M$ of $GL_n(q)$ of type $8$. Then $M\cap SL_n(q)$ is an almost quasi-simple classical group of dimension $n$. Thus $H\cap SL_n(q)$ must be a strict subgroup of $M\cap SL_n(q)$ and we take a further subgroup, appealing to the results of \cite{kl}. Repeating if necessary we find that $H\cap SL_n(q)$ is a subgroup of $M\cap SL_n(q)$ of type 1 to 7; then $H$ must lie inside a subgroup of that type in $GL_n(q)$.

Now there are no subgroups of type 8 in $GL_2(q)$. Thus in what follows we can assume that $H$ lies inside a maximal subgroup of type 1 to 7, unless $n=3,4,5$ or $6$.

\begin{lemma}\label{L: oddgln}
Suppose $H<GL_n(q)$ and $H$ has odd order. If $n\geq 2$ then $|H|_{p',\heartsuit}<q^{\frac{3n-2}2}.$
\end{lemma}
\begin{proof}
We will prove this statement using induction on $n$. For $n=2$ the statement is easily checked. Now assume inductively that the statement is true for $H<GL_m(q)$ where $1<m<n$. 

First observe that if $H$ lies in a maximal subgroup of type 5 in Aschbacher's set of families then $H<GL(n,q_0)\circ Z(GL(n,q))$ where $q_0=p^b$. Then it is sufficient to prove our Lemma for $q=p^b$. Thus we may assume that $H$ does not lie in a maximal subgroup of type 5.\label{two}

If $H$ lies in a parabolic subgroup of $GL_n(q)$ then $H/O_p(H)<GL_m(q)\times GL_{n-m}(q)$. If both $m$ and $n-m$ are greater than $1$ then induction implies that 
$$|H|_{p',\heartsuit}<q^{\frac{3m-2}2}\times q^\frac{3(n-m)-2}2=q^{\frac{3n-4}2}$$
which is sufficient. On the other hand suppose that $m=1$. Then 
$$ |H|_{p',\heartsuit}<q\times q^\frac{3(n-1)-2}2=q^{\frac{3n-3}2}$$
which is sufficient. If $H$ lies in a maximal subgroup $M$ of $GL_n(q)$ of type 4 then $H<GL_m(q)\times GL_t(q)$ for $n=mt$. A similar inductive argument gives the result. \label{five}

If $H$ lies in a maximal subgroup $M$ of $GL_n(q)$ of type 2 then $M=GL_m(q)\wr S_t$ where $n=mt, t\geq 2$. We assume that $H$ acts transitively on the subspace decomposition otherwise $H$ lies in a parabolic subgroup. In fact we assume that $H$ acts primitively on the subspace decomposition since otherwise $H$ lies in a maximal subgroup $M_1=GL_{m_1}(q)\wr S_{t_1}$ of type $2$ with $t_1<t$ and $H$ acts primitively on this decomposition. 

If $m>1$ then, by induction and by Lemma \ref{L: oddsn},
$$|H|_{p',\heartsuit}<(q^{\frac{3m-2}2})^tt^{\log_2t}.$$\label{three}
This is enough to give our result. If $m=1$ then $|H|_{p',\heartsuit}<(\frac{q-1}2)^n. n^{\log_2 n}$ which gives the result.

If $H$ lies in a maximal subgroup $M$ of $GL_n(q)$ of type 3 then $M=GL_m(q^r).r$ where $n=mr$. If $m>1$ then, by induction,
$$|H|_{p',\heartsuit}<((q^r)^{\frac{3m-2}2}).|r|_{2'}$$ where $n=mr$, $r\geq 2$. This is enough to give our result. If $m=1$ then $M=(q^n-1).n$ and the result follows.\label{four} 

If $H$ lies in a maximal subgroup $M$ of $GL_n(q)$ of type 6 then Lemma \ref{L: six} gives the result. If $H$ lies in a maximal subgroup $M$ of $GL_n(q)$ of type 7 then Lemma \ref{L: seven} implies that $H<GL_{\sqrt{n}}(q)\times GL_{\sqrt{n}}(q)$. In this case induction implies that 
$$|H|_{p',\heartsuit}<(q^{\frac{3\sqrt{n}-2}2})^2 = q^{3\sqrt{n}-2}<q^{\frac{3n-2}2}$$
as required.

If $H$ lies inside a group $M$ of type 8 then we need only consider the possibility that $3\leq n\leq 6$. We refer to \cite[Proposition 2.9.1]{kl}. Consider first the following cases where a simple group $K$ is involved in $M$: We have $n=3$ and $K=\Omega_3(q)\cong PSL_2(q)$; we have $n=5$ and $K=\Omega_5(q)\cong PSp_4(q)$; we have $n=6$ and $K=\Omega_6^\epsilon(q)\cong PSL_4^\epsilon(q)$. In all these cases $K$ is involved in a maximal subgroup of $GL_m(q)$ for some $m$ smaller than $n$. What's more $|M|_{p',\heartsuit}\leq \frac{(n,3)}{2}(q-1)|K|_{p',\heartsuit}$. Then, by induction, $|K|_{p',\heartsuit}<q^{\frac{3m-2}2}$ and so $|M|_{p',\heartsuit}<q^{\frac{3m-2}2}$ as required.

The only remaining case is when $n=4$ and $K=\Omega_4^\epsilon(q)$. In this case $|K|_{p',\heartsuit}<q^4$ and $|M|_{p',\heartsuit}<q|K|_{p',\heartsuit}<q^5$ as required.
\end{proof}

\begin{lemma}\label{L: second gap}
Let $H$ be an odd order subgroup of $GL_n(q)$ and let $\sigma$ be an involutory field automorphism of $GL_n(q)$. Suppose that $H$ is normalized by $g$, an involution in $\langle GL_n(q), \sigma\rangle$. Then
$$|H:C_H(g)|_{p', \heartsuit}< q^{\frac{3n-1}{4}}.$$
If $n=2$ then we can strengthen this to give $|H:C_H(g)|_{p', \heartsuit}< q.$
\end{lemma}
\begin{proof}
Note that, by \cite[Proposition 1.1]{bgl}, we know that $g$ is $GL_n(q)$-conjugate to $\sigma$ hence we take $H$ to be normalized by $\sigma$. Observe that $C_{Z(G)}(g)=(\sqrt{q}-1)$ and that $q\geq 49$. 

As before we will prove the statement by induction on $n$. We check the base case: For $n=1$ it is clear that $|H:C_H(g)|<\frac{\sqrt{q}+1}2$. For $n=2$ we can prove the stronger bound of $q$ directly - we omit the details.\label{eight} 

If $\langle H,\sigma\rangle$ lies inside a maximal subgroup of $\langle GL_n(q), \sigma\rangle$ of type 5 then
$$|H:C_H(g)|_{p', \heartsuit}\leq \frac{\sqrt{q}+1}{2}|GL_n(q_0):C_{GL_n(q_0)}(g)|_{p', \heartsuit}.$$
Thus it is enough to prove the result for $GL_n(q_0)$ where $q_0$ is the order of the subfield. Hence we assume that $H$ does not lie inside a subgroup of type 5.

For $\langle H,\sigma\rangle$ inside a group of type $6$ or $7$ then Lemmas \ref{L: six} and \ref{L: seven} imply that $H<N=GL_{\sqrt{n}}(q)\times GL_{\sqrt{n}}(q)$. $N$ is normalized in $GL_n(q)$ by $\tau$, an involution which swaps the two copies of $GL_{\sqrt{n}}(q)$. Thus $g$ may take two forms.

Firstly suppose that $g=(A,B)\sigma$ where $A,B\in GL_{\sqrt{n}}(q)$. Then, by \cite[Proposition 1.1]{bgl}, $g$ is $N$-conjugate to $\sigma$ and induction implies that
$$|H:C_H(\sigma)|_{p', \heartsuit}\leq (q^{\frac{3\sqrt{n}-1}{4}})^2<q^{\frac{3n-1}{4}}.$$

Secondly suppose that $g=(A,B)\tau\sigma$. Then, for $(X,Y)\in N$, 
$$(X,Y)^g = (AY^\sigma A^{-1}, A^{-\sigma}X^\sigma A^\sigma).$$
Thus $(X,Y)$ will be centralized by $g$ if and only if $X=AY^\sigma A^{-1}$. Thus $|N:C_N(g)|=|GL_{\sqrt{n}}(q)|$ and so $|H:C_H(g)|_{p', \heartsuit}\leq |K|$ where $K$ is a subgroup of odd order in $GL_{\sqrt{n}}(q)$. Hence, by Lemma \ref{L: oddgln}, $|H:C_H(g)|_{p', \heartsuit}<q^{\frac{3\sqrt{n}-2}2}<q^{\frac{3n-1}4}$ as required.

If $\langle H,\sigma\rangle$ lies inside a subgroup of type 1 then
$$|H:C_H(g)|_{p', \heartsuit} = |L:C_L(g_L)|_{p',\heartsuit}$$
where $L\cong H/O_p(H)$ (here $O_p(H)$ is the largest normal $p$-group of $H$). Then $L\leq GL_m(q)\times GL_{n-m}(q)$ and we may examine $GL_m(q)$ and $GL_{n-m}(q)$ separately. In each case $g_L$ acts as an involution in $\langle GL_k(q), \sigma \rangle$ hence we can apply induction:
$$|H:C_H(g)|_{p', \heartsuit}< q^{\frac{3m-1}{4}}\times q^{\frac{3(n-m)-1}{4}}< q^{\frac{3n-1}{4}}.$$

A similar approach can be taken if $\langle H,\sigma\rangle$ lies inside a subgroup $M$ of type 4. Then $M\cong GL_m(q)\circ GL_{t}(q)$ for $n=mt$ and, once more, $g$ acts as an involution in $\langle GL_k(q), \sigma \rangle$. Induction gives the result.

Now suppose that $H\leq M\cong GL_{\frac{n}{r}}(q^r).r$, $r$ prime, a subgroup of type 3 in $GL_n(q)$. If $r=2$ then any element from $\langle GL_n(q),\sigma\rangle \backslash GL_n(q)$ which normalizes $M$ will act as a field automorphism of order $4$ on $GL_{\frac{n}{r}}(q^r)$. In particular such an element cannot be an involution. On the other hand if $r$ is odd then $\langle M, \sigma\rangle\cong GL_{\frac{n}{r}}(q^r).2r$. Thus $g$ must act as an involutory field automorphism on $GL_{\frac{n}{r}}(q^r)$. Then Lemma \ref{L: oddnormal} implies that
$$|H:C_H(g)|_{p', \heartsuit} = |H\cap GL_{\frac{n}{r}}(q^r):C_{H\cap GL_{\frac{n}{r}}(q^r)}(g)|_{p', \heartsuit}$$
and induction gives the result.

Next consider the possibility that $\langle H,\sigma\rangle$ lies inside a subgroup of type 2. Thus $H<\langle(GL_m(q)\wr S_t),\sigma\rangle$. Just as for Lemma \ref{L: oddgln}, we assume that $\langle H,\sigma\rangle$ acts primitively on the subspace decomposition. Take $g=s\sigma$ and note that $s$ acts as an involution on the subspace decomposition.\label{ten} 

We need to consider two situations which closely mirror the two cases discussed for a subgroup of type 7. First consider $C_{S}(g)$ where $S$ is the projection of $H$ onto a particular $GL_m(q)$ which is fixed by $s$. By induction $|S: C_{S}(g)|_{p', \heartsuit}<q^{\frac{3m-1}{4}}.$ Alternatively if $GL_m(q)\times GL_m(q)$ are swapped by $s$, and $S$ is the projection of $H$ onto $GL_m(q)\times GL_m(q)$ then it is clear that $|S: C_{S}(g)|_{p',\heartsuit}$ is at most the size of an odd-order subgroup in $GL_m(q)$. Thus, by Lemma \ref{L: oddgln}, this is bounded above by $q^{\frac{3m-2}{2}}$.

Now write $N=\underbrace{GL_m(q)\times\dots\times GL_m(q)}_t$. Then Lemma \ref{L: oddnormal} implies that 
$$|H:C_H(g)|=|H\cap N: C_{H\cap N}(g)||H/N: C_{H/N}(gN)|.$$
Thus, writing $s$ as the product of $k$ transpositions in its action on the subspace decomposition, we have
\begin{eqnarray*}
|H\cap N: C_{H\cap N}(g)|_{p', \heartsuit}&<&(q^{\frac{3m-1}{4}})^{t-2k}\times (q^{\frac{3m-2}{2}})^k \\
&<& q^{\frac{3n-1}{4}} \times \frac{1}{q^{\frac{t-1}{4}}}.
\end{eqnarray*}
Then, referring to Lemma \ref{L: oddsn}, it is sufficient to prove that $t^{\log_2t}<49^{\frac{t-1}{4}}$ which is clear.

If $H$ lies inside a group $M$ of type 8 then, once again, we need only consider the possibility that $3\leq n\leq 6$. In these cases we proceed similarly to the proof in Lemma \ref{L: oddgln} by appealing to \cite[Proposition 2.9.1]{kl}. Consider first the following cases where a simple group $K$ is involved in $M$: We have $n=3$ and $K=\Omega_3(q)\cong PSL_2(q)$; we have $n=5$ and $K=\Omega_5(q)\cong PSp_4(q)$; we have $n=6$ and $K=\Omega_6^\epsilon(q)\cong PSL_4^\epsilon(q)$. In all these cases $K$ is involved in a maximal subgroup of $GL_m(q)$ for some $m$ smaller than $n$. Let $H_K$ be the section of $H$ lying in $K$. Then $g$ induces an involutory field automorphism $g_K$ on $H_K$ hence, by induction, 
$|H_K:C_{H_K}(g_k)|_{p',\heartsuit}<q^{\frac{3m-1}4}$. Furthermore $|M/(M\cap Z(GL_n(q)))|_{p',\heartsuit}\leq (n,3)|K|_{p',\heartsuit}$. Hence we conclude that
$$|H:C_H(g)|_{p',\heartsuit}<|\sqrt{q}+1|_{p',\heartsuit}(n,3)q^{\frac{3m-1}4}$$
as required.

The only remaining case is when $n=4$ and $K=P\Omega_4^\epsilon(q)$ is involved in $M$. Note that $P\Omega_4^+(q)\cong PSL_2(q)\times PSL_2(q)$ and $P\Omega_4^-(q)\cong PSL_2(q^2)$ so $K$ is not necessarily simple. \label{nine}Now $g$ induces an involutory field automorphism $g_K$ on $H_K$. Hence, by our strengthened bound for $n=2$, 
$|H_K:C_{H_K}(g_k)|_{p',\heartsuit}<q^2$. Furthermore $|M/(M\cap Z(GL_n(q)))|_{p',\heartsuit}\leq |K|_{p',\heartsuit}$. Hence we conclude that
$$|H:C_H(g)|_{p',\heartsuit}<\frac{\sqrt{q}+1}2q^2<q^{\frac52}$$
as required.
\end{proof}

\begin{lemma}\label{L: reduced bound}
Let $H<GL_n(q)$ with $|H|$ odd. Suppose that $g$ is an involution in $GL_n(q)$ which normalizes $H$. Then,
$$|H:C_H(g)|_{p', \heartsuit}<q^{\frac{3n-2}4}.$$
\end{lemma}
\begin{proof}
We proceed by induction on $n$. When $n=1$ the result is trivial and, when $n=2$, Lemma \ref{L: base case} gives the result. Now suppose that $H\leq M$, a maximal subgroup of $GL_n(q)$. Our method of proof mirrors the techniques used to prove Lemmas \ref{L: oddgln} and \ref{L: second gap}.

If $M$ is of type 5 then we simply descend to the base field and continue. If $M$ is of type $6$ or $7$ then Lemmas \ref{L: six} and \ref{L: seven} gives us the result. If $M$ is of type 1 or 4 then an easy inductive result gives us the result (as before).

If $M$ is of type $2$ then $H<GL_m(q)\wr S_t$. Set $N=\underbrace{GL_m(q)\times\dots\times GL_m(q)}_t$. If $g$ lies inside $N$ then induction implies that $|H\cap N:C_{H\cap N}(g)|_{p',\heartsuit}<(q^{\frac{3m-2}4})^t$. 

If $g\in H\backslash N$, then we proceed very similarly to the proof of Lemma \ref{L: second gap}. First consider $C_{S}(g)$ where $S$ is the projection of $H$ onto a particular $GL_m(q)$ which is fixed by $g$. By induction $|S: C_{S}(g)|_{p', \heartsuit}<q^{\frac{3m-2}{4}}.$ Alternatively if $GL_m(q)\times GL_m(q)$ are swapped by $g$, and $S$ is the projection of $H$ onto $GL_m(q)\times GL_m(q)$ then it is clear that $|S: C_{S}(g)|_{p',\heartsuit}$ is at most the size of an odd-order subgroup in $GL_m(q)$. Thus, by Lemma \ref{L: oddgln}, this is bounded above by $q^{\frac{3m-2}{2}}$. Thus $|H\cap N:C_{H\cap N}(g)|_{p',\heartsuit}<(q^{\frac{3m-2}4})^t$. In both cases we have $|H:C_H(g)|_{p',\heartsuit}< q^{\frac{3n-2t}4}\times t^{\log_2t}$ which is sufficient.

If $M$ is of type $3$ and $H<GL_m(q^r).r$, where $n=mr$ with $r$ prime, then there are two cases. If $g$ lies inside $GL_m(q^r)$ then induction implies that
$$|H:C_H(g)|_{p', \heartsuit}<(q^r)^{\frac{3m-2}4}r<q^{\frac{3n-2}4}.$$
. Otherwise if $g$ does not lie in $GL_m(q^r)$ then we must have $r=2$ and Lemma \ref{L: second gap} gives the result.

Finally if $H$ lies inside a group $M$ of type 8 then, once again, we need only consider the possibility that $3\leq n\leq 6$. As before, for $n\neq 4$ we simply use the fact that the simple classical group involved in $M$ is involved in a maximal subgroup of type 8 in $GL_m(q)$ for some $m<n$. Then induction gives the result. 

If $n=4$ and $P\Omega^\epsilon_4(q)$ is involved in $M$ then we make use of the isomorphisms given in \cite[Proposition 2.9.1]{kl} and refer to Lemma \ref{L: base case}.
\end{proof}

\subsection{The inductive part}

Now we suppose that Lemma A holds for $m<n$, i.e. we proceed under the following hypothesis:
\begin{ihypothesis}
Let $H_1$ be an even order subgroup of $GL_m(q)$ with $m<n$. Then $H_1$ contains an involution $g$ such that
$|H_1:C_{H_1}(g)|_{p',\heartsuit}\leq q^{m-1}+\dots+q+1$.
\end{ihypothesis}

\begin{lemma}
If $H$ has even order and lies in a maximal subgroup of type 1, 3 or 4 then 
$|H:C_H(g)|_{p',\heartsuit}\leq q^{n-1}+\dots+q+1$.
\end{lemma}
\begin{proof}
If $H$ lies in a maximal subgroup of type 1 then $H\leq Q:(GL_m(q)\times GL_{n-m}(q)$ where $Q$ is a $p$-group and $m>1$. Then 
$$|H:C_H(g)|_{p',\heartsuit}\leq (q^{m-1}+\dots+q+1)(q^{n-m-1}+\dots+q+1)<q^{n-1}+\dots+q+1.$$

If $H$ lies in a maximal subgroup of type 3 then $H\leq M=GL_m(q^r).r$ where $n=mr$ and $r$ is prime. Let $N=GL_m(r)$ be normal in $M$ and split into two cases. Suppose first that $|H\cap N|$ is even. Then induction implies that $H\cap N$ contains an involution $g$ such that
$$|H:C_H(g)|_{p',\heartsuit}\leq ((q^r)^{m-1}+\dots+q+1)r< q^{mr-1}+\dots+q+1$$
as required. Suppose on the other hand that $|H\cap N|$ is odd. Then $r=2$ and  Lemma \ref{L: second gap} gives the result.

If $H$ lies in a maximal subgroup of type 4 then $H\leq GL_m(q)\circ GL_r(q)$ where $n=mr$ and $m,r>1$. Then induction implies that
$$|H:C_H(g)|_{p',\heartsuit}\leq (q^{m-1}+\dots+q+1)(q^{r-1}+\dots+q+1)<q^{mr-1}+\dots+q+1.$$
\end{proof}

We are left with subgroups lying only in maximal subgroups of type 2. 

\subsubsection{Maximal subgroups of type 2}

Take $H\leq M=GL_m(q)\wr S_t, n=mt, t\geq2.$ Write 
$$N=\underbrace{GL_m(q)\times\dots\times GL_m(q)}_t.$$ 

It is clear that we can assume that $M$ acts primitively on the vector decomposition, otherwise we lie in a different subgroup of type 2 which we consider instead. We have three subcases, the first of which is dealt with in the following lemma.

\begin{lemma}
Suppose that $H\cap N$ has odd order and $t<n$. Then $H$ contains an involution $g$ such that
$|H:C_H(g)|_{p',\heartsuit}\leq q^{n-1}+\dots+q+1$.
\end{lemma}
\begin{proof}
By Lemma \ref{L: oddgln},
$$|H\cap N|_{p',\heartsuit}<(q^{\frac{3m-2}{2}})^t.$$
Then, for $g$ an involution in $H$,
$$|H\cap N:C_{H\cap N}(g)|_{p',\heartsuit}<(q^{\frac{3m-2}{4}})^t\leq q^{n-t}.$$
Now $S_t<GL_t(q)$ hence, by induction, there is an involution $g$ in $H$ such that $|H/N:C_{H/N}(gN)|<q^{t-1}+\dots+q+1.$ Thus, by Lemma \ref{L: oddnormal},
$$|H:C_H(g)|<q^{n-t}(q^{t-1}+\dots+q+1)<q^{n-1}+\dots+q+1.$$
\end{proof}

We next consider the possibility that $H\cap N$ has odd order and $t=n$. We state a preliminary lemma before dealing with this case.

\begin{lemma}\label{L: sninvolutions}
If $H$ is a primitive subgroup of $S_n$  of even order then $H$ contains an involution $g$ such that
$|H:C_H(g)|<42^{\frac{n-2}{2}}$.
\end{lemma}
\begin{proof}
The statement is clearly true if $n\leq 7$. If $n\geq 8$ then $4^n<42^{\frac{n-2}{2}}$. Then we refer to \cite{ps} which asserts that any primitive subgroup of degree $n$ must either contain $A_n$ or else has order less than $4^n$.

If $H$ contains $A_n$ then $H$ contains an involution $g$ such that $|H:C_H(g)|<n^4$ and we are done.
\end{proof}

\begin{lemma}
Suppose that $H\cap N$ has odd order and $t=n$. Then $H$ contains an involution $g$ such that
$|H:C_H(g)|_{p',\heartsuit}\leq q^{n-1}+\dots+q+1$.
\end{lemma}
\begin{proof}
Observe that $N=(q-1)^n$. We proceed very similarly to the proofs of Lemmas \ref{L: second gap} and \ref{L: reduced bound}. First consider $C_{S}(g)$ where $S$ is the projection of $H$ onto a particular $q-1$ which is fixed by $g$. Clearly $|S: C_{S}(g)|_{p', \heartsuit}=1.$ Alternatively if $(q-1)\times (q-1)$ are swapped by $g$, and $S$ is the projection of $H$ onto $(q-1)\times(q-1)$ then it is clear that $|S: C_{S}(g)|_{p',\heartsuit}< (q-1)$. Thus, in all cases,
$$|N:C_N(g)|_{p',\heartsuit}<(q-1)^{\frac{n}{2}}.$$
By Lemma \ref{L: sninvolutions}
$$|H/N:C_{H/N}(gN)|\leq 42^\frac{n-2}{2}.$$
By Lemma \ref{L: oddnormal},
$$|H:C_H(g)|_{p',\heartsuit}< (q-1)^{\frac{n}{2}}\times 42^\frac{n-2}{2}.$$
For $q\geq 43$ this is less than $q^{n-1}+\dots+q+1.$

If $q<43$ then $q=7,13,19,31$ or $37$. In all cases $|q-1|_{p',\heartsuit}=3$ and hence $|N:C_N(g)|_{p',\heartsuit}\leq3$. Thus, by Lemma \ref{L: oddnormal}, 
$$|H:C_H(g)|_{p',\heartsuit}\leq 3\times 42^\frac{n-2}{2}\leq q^{n-1}+\dots+q+1.$$

\end{proof}

The final subcase is when $H\cap N$ has even order. The following lemma deals with this situation.

\begin{lemma}
Suppose that $H\leq GL_m(q)\wr S_t$ where $n=mt$. Write 
$$N=\underbrace{GL_m(q)\times\dots\times GL_m(q)}_t.$$
Suppose that $H\cap N$ has even order. Then $H$ contains an involution $g$ such that
$$|H:C_H(g)|_{p', \heartsuit}\leq q^{n-1}+\dots+q+1.$$
\end{lemma}
\begin{proof}
We begin by investigating the quantity $|H\cap N:C_{H\cap N}(g)|_{p',\heartsuit}$. Our analysis (and our notation) mirrors the set up in Lemma \ref{L: invcentralizer}. Write $N=N_1\times \dots \times N_t$ with $N_i\cong GL_m(q)$ for $i=1,\dots,t$. Write $L_i$ for the projection of $H\cap N$ onto $N_i\times\dots\times N_t$ and write $\psi_i:L_i\to L_{i+1}$ for the canonical projection. Let $T_i$ be the kernel of $\psi_i$ for $1\leq i<t$; define $T_t:= L_t$. Suppose that $|T_i|$ is odd for $i<k\leq t$ and $|T_i|$ is even for $i=k$.

Observe that $T_k\leq GL_m(q)$. Then, by induction, take an involution $g_k\in T_k$ such that 
$$|T_k:C_{t_k}(g_k)|_{p', \heartsuit}\leq q^{m-1}+\dots+q+1.$$
Let $g\in G$ be the pre-image of $g_k$ in $H\cap N$. Write $g_i$ for the image of $g$ under the projection of $H\cap N$ onto $L_i$. Then Lemma \ref{L: oddgln} implies that $|T_i:C_{T_i}(g_i)|_{p',\heartsuit}<q^{\frac{3m-2}4}$ for $1\leq i <k$. Thus Lemma \ref{L: invcentralizer} implies that
\begin{eqnarray}
|H\cap N:C_{H\cap N}(g)|_{p',\heartsuit}<q^{(\frac{3m-2}4)(k-1)}(q^{m-1}+\dots+q+1)\frac{|(g_k)^{L_k}\cap P_k|_{p',\heartsuit}}{|(g_k)^{T_k}\cap P_k|_{p',\heartsuit}}.\label{eqn}
\end{eqnarray}
Here $P_k$ is a Sylow $2$-subgroup of $T_k$ and so $(g_k)^{L_k}\cap P_k$ is the set of $L_k$-conjugates of $g_k$ which lie in $P_k$.

If $k=t$ then $L_k=T_k$; if $k<t$ then Lemma \ref{L: sylowtwoingln} implies that $|(g_k)^{L_k}\cap P_k|\leq q^{m-1}+\dots+q+1$. Thus in all cases it is clear that
\begin{eqnarray}
|H\cap N:C_{H\cap N}(g)|_{p',\heartsuit}\leq (q^{m-1}+\dots+q+1)^{\min\{k+1,t\}}.\label{eqnb}
\end{eqnarray}
Now $H/N\leq S_t$ and so $H$ can fuse at most $t!$ conjugacy classes of $H\cap N$ into one. Thus we have
$$|H:C_{H}(g)|_{p',\heartsuit}\leq (q^{m-1}+\dots+q+1)^tt!$$
Now if $t<20$ then $(t!)_{p', \heartsuit}<(p-1)^{t-1}$ which implies that $|H:C_H(g)|_{p', \heartsuit}\leq q^{n-1}$ as required. Thus we assume that $t\geq 20$.

For the remainder we split into two cases. First suppose that $k\geq 5$. Recall that $P_k$ is a Sylow $2$-subgroup of $T_k$ and let $P$ be a Sylow $2$-subgroup of $H\cap N$. Observe that $P_k\cong P$. Furthermore $P$ is a Sylow $2$-subgroup of $H$, while $P_k$ is a Sylow $2$-subgroup of $L_k\cong (H\cap N)/K$ where $K$ is an odd order subgroup. In particular this implies that
\begin{eqnarray}
|(g_k)^{L_k}\cap P_1| = |g^{H\cap N}\cap P|.\label{eqnc}
\end{eqnarray}
Then (\ref{eqn}), (\ref{eqnc}) and Lemma \ref{L: sylowtwos} imply that
$$|H:C_H(g)|_{p',\heartsuit}<q^{(\frac{3m-2}4)(k-1)}(q^{m-1}+\dots+q+1)|g^H\cap P|_{p',\heartsuit}.$$

Observe next that $P$ is isomorphic to a $2$-group lying in $GL_m(q)^{t-k+1}$. Thus $|g^H\cap P|$ is less than the total number of involutions in a Sylow $2$-group of $GL_m(q)^{t-k+1}$.  Thus Lemma \ref{L: sylowtwoingln} implies that $|g^H\cap P|<(q^{m-1}+\dots+q+3)^{t-k+1}$. Hence we have
$$|H:C_H(g)|_{p',\heartsuit}<q^{(\frac{3m-2}4)(k-1)}(q^{m-1}+\dots+q+1)(q^{m-1}+\dots+q+3)^{t-k+1}<q^{n-1}$$ as required.

Suppose that $k\leq 4$. Now $H$ acts by permuting elements in $N$, and $g$ has form 
$$(h_1,\dots, h_k,1,\dots,1).$$
 Thus $H$ must fuse at most $k!{t\choose k}$ such conjugacy classes of $H\cap N$. For $k\leq 4$ we have $k!{t\choose k}< t^4$ hence, by (\ref{eqnb}) we have
$$|H:C_H(g)|< (q^{m-1}+\dots+q+1)^5t^4<q^{5m-1}t^4.$$
Now, for $t\geq 20$, this is less than $q^{mt-1}+\dots+q+1$ as required.
\end{proof}

This completes our proof of Lemma A, and thereby proves Theorem A.

\section{Quaternion Sylow $2$-subgroups}\label{S: quaternionsylowtwos}

In this section we operate under Hypothesis \ref{h:basic} and use the associated notation, as defined at the start of Section \ref{SS: pp}. We will also use the notation that, for a group $H$, $\bar{H}=H/O(H)$ where $O(H)$ is the largest odd-order normal subgroup of $H$.

We are interested in the possibility that $G$ is insoluble. Now Theorem A implies that in this case $G$ must have generalized quaternion Sylow $2$-subgroups. Hence, for the remainder of this paper, we suppose that this is the case. 

Using the comments in \cite[p377]{gorenstein3} we can adapt the results in \cite{gw1} to write the structure of $\bar{G}$. Then $\bar{G}$ is isomorphic to one of the following groups:
\begin{itemize}
\item a Sylow-2 subgroup of $G$.
\item $Z.A_7$, a central extension of $A_7$ by $Z$, a group of order $2$.
\item $SL_2(q).D$ where $D$ is cyclic of odd order $d$.
\end{itemize}

Note that, in particular, $G=O(G)C_G(g)$ for $g$ an involution. Next we note a background result:



\begin{lemma}\cite[(5.21)]{aschbacher3}
Let $G$ be transitive on a set $X$, $x\in X$, $H$ the stabilizer in $G$ of $x$, and $K\leq H$. Then $N_G(K)$ is transitive on $\Fix(K)$ if and only if $K^G\cap H=K^H$.
\end{lemma}

Given $G$ with quaternion Sylow $2$-subgroups we will take $K$ to be a subgroup of $\Galph$ of order $2$, say $K=\{1,g\}$. Then $g$ is the unique involution in a Sylow $2$-subgroup of $\Galph$ and so $K^G\cap H=K^H$. Hence $N_G(K)=C_G(g)$ is transitive on $\Fix(K)=\Fix(g)$, a Baer subplane of $\spaceP$. Then we have the following proposition (which implies Theorem B):

\begin{proposition}\label{P: quaternion}
Suppose that $G$ acts transitively on the set of points of $\spaceP$, a non-Desarguesian projective plane. If $G$ has generalized quaternion Sylow $2$-subgroups then $G$ contains an odd-order subgroup $H$ which is also transitive on the set of points of $\spaceP$.
\end{proposition}
\begin{proof}
We have listed three possibilities for the structure of $G$. The proposition clearly holds for the first possibility, so exclude this. Take $g\in G$ an involution and let $C^0$ be the kernel of the action of $C_G(g)$ upon $\Fix(g)$.

Suppose first that $\Fix(g)$ is Desarguesian. There are two possibilities for $C_G(g)/C^0$: Either $C_G(g)/C^0$ is soluble or $C_G(g)/C^0$ has socle $PSL_3(u)$ where $u>2$ is the order of $\Fix(g)$. Now, since $u>2$, $PSL_3(u)$ has Sylow $2$-subgroups which are neither cyclic nor dihedral and so they cannot form a section of a quaternion group. Hence we may conclude that $C_G(g)/C^0$ is soluble.

Now suppose that $\Fix(g)$ is not Desarguesian. Then $C_G(g)/C^0$ must have dihedral or cyclic Sylow $2$-subgroups. The former is not possible by Theorem A and the latter implies that $C_G(g)/C^0$ is soluble. Hence this conclusion holds in all cases.

Since $G=O(G)C_G(g)$ and $G$ is insoluble we conclude that $C_G(g)$ is insoluble; in particular $C^0$ is insoluble. Furthermore $C^0/(C^0\cap O(G))$ is isomorphic to a normal subgroup of $\bar{G}$.

If $\bar{G}\cong Z.A_7$ then the only insoluble normal subgroup of $\bar{G}$ is $\bar{G}$. Hence $C^0/(C^0\cap O(G))=\bar{G}$. Thus $G=C^0 O(G)=G_\alpha O(G)$ and so $O(G)$ is transitive on the set of points of $\spaceP$.

If $\bar{G}\cong SL_2(q).D$ then an insoluble normal subgroup of $\bar{G}$ must contain $SL_2(q)$. Thus $C^0/(C^0\cap O(G))\geq SL_2(q)$. Now pick an element, $h$, in $G$ whose image in $\bar{G}/SL_2(q)$ has order divisible by $D$; it is clear that such an element can be chosen to have odd order. Then $G=C^0 \langle O(G), h \rangle = G_\alpha \langle O(G), h \rangle$ and so $\langle O(G), h \rangle$ is transitive on the set of points of $\spaceP$.
\end{proof}

\bibliographystyle{amsalpha}
\bibliography{paper}

\end{document}